\documentclass[journal,10pt,onecolumn,draftclsnofoot]{IEEEtran}  

\IEEEoverridecommandlockouts                              

\usepackage{amsmath, amsfonts, dsfont, amssymb, mathrsfs, setspace, graphicx, epsfig, amsthm,url,float,framed,circuitikz, color,subcaption,epstopdf}
\usepackage{multicol}
\usepackage{hyperref,kbordermatrix,tikz}
\usetikzlibrary{shapes.gates.logic.US,trees,positioning,arrows,quotes}

\newtheorem{theorem}{Theorem}

\newtheorem{definition}{Definition}
\newtheorem{lemma}{Lemma}
\newtheorem{corollary}{Corollary}

\newtheorem{prop}{Proposition}
\newtheorem{remark}{Remark}
\newtheorem{observation}{Observation}
\newtheorem{assumption}{Assumption}

\makeatletter
\usepackage[noadjust]{cite}
\makeatother

\def\scr#1{{\cal #1}}
\newcommand{\R}{\mathbb{R}}

\newcommand{\0}{\mathbf{0}}
\newcommand{\1}{\mathbf{1}}

\title{Analysis, Identification, and Validation of Discrete-Time Epidemic Processes}

\author{Philip E. Par\'{e}, Ji Liu, Carolyn L. Beck, Barret E. Kirwan, and Tamer Ba\c sar
*\thanks{
* Philip E. Par\'{e}, Tamer Ba\c sar, and Carolyn L. Beck are with the Coordinated Science Laboratory at the University of Illinois at Urbana-Champaign (UIUC) and can be reached at {\tt philip.e.pare@gmail.com}, {\tt basar1@illinois.edu}, and {\tt beck3@illinois.edu}, respectively. 
Ji Liu is with the Department of Electrical and Computer Engineering at Stony Brook University and can be reached at {\tt ji.liu@stonybrook.edu}. 
Barrett Kirwan is with the Agricultural and Consumer Economics Department at UIUC. 
 This material is based on research partially sponsored by the USDA, CA 58-6000-4-0028, and the National Science Foundation, grant CPS 1544953.  All material in this paper represents the position of the authors and not necessarily that of the NSF or the USDA.
}}
\begin{document}
\maketitle

\begin{abstract}
Models of spread processes over non-trivial networks are commonly motivated by modeling and analysis of biological networks,  computer networks, and human contact networks. 
However, identification of such models has not yet been explored in detail, and the models have not been validated by real data. 
In this paper, we present several different spread models from the literature and explore their relationships to each other; for one of these processes, we present a sufficient condition for asymptotic stability of the healthy equilibrium, show that the condition is necessary and sufficient for uniqueness of the healthy equilibrium, and present necessary and sufficient conditions for learning the spread parameters. Finally, we employ two real datasets, one from John Snow's seminal work on cholera epidemics in London in the 1850's and the other one from the United States Department of Agriculture, to validate an approximation of a well-studied network-dependent susceptible-infected-susceptible (SIS) model. 
\end{abstract}

\section{Introduction}

Spread processes have been studied in many fields. In the systems and control community, the main interest has been on SIS spread models over non-trivial networks. These models have been proposed for discrete time \cite{wang2003epidemic,chakrabarti2008epidemic,ahn2013global} and continuous time \cite{ahn2013global,van2009virus,preciado2014optimal,pare2015stability,pare2017epidemic}, and are based on an infection parameter $\beta$ and a healing rate $\delta$. A virus model is called \textit{homogeneous} if the infection and healing rates are the same for every agent, and \textit{heterogeneous} if they are different for each agent. 
In this work, we will focus  on discrete-time SIS  models.


In \cite{wang2003epidemic}, Wang {\em et al.} introduce a discrete-time homogeneous virus spread model that is dependent on a nontrivial undirected graph structure. The authors give an epidemic threshold for the model in terms of the maximum eigenvalue of the matrix depicting the graph structure in relation to the ratio of $\beta$ and $\delta$ that ensures convergence to the healthy state, that is, where the virus is eradicated. In \cite{van2009virus}, Van Mieghem {\em et al.} point out that the model in \cite{wang2003epidemic} is only accurate for spread processes if the virus is being eradicated. 
In \cite{chakrabarti2008epidemic}, Chakrabarti {\em et al.} explore the same model as \cite{wang2003epidemic} but in more detail. 
Ahn and Hassibi, in \cite{ahn2013global},  study both discrete- and continuous-time homogeneous SIS models.  Both the healthy and the endemic states of several models are considered, and existence, uniqueness, and stability conditions for special cases of the endemic state are established. They also provide a sufficient condition for global stability of the endemic state for the model in \cite{wang2003epidemic,chakrabarti2008epidemic}.

The model we focus on in this work is similar to a special case in \cite{ahn2013global}. However, the model in \cite{ahn2013global} assumes homogeneous virus spread and an unweighted adjacency matrix. The models in Sections \ref{sec:mod} and \ref{sec:analysis} are not limited by these assumptions. 

While parameter estimation of epidemic spread with real data has been carried out for some models \cite{keeling2001dynamics,miao2011identifiability,kolesnichenko2016codered}, the previous work has either not had network structure included or employed a large probabilistic model. Ignoring network structure is clearly a huge assumption and using a full probabilistic model can become very computationally expensive as the size of the network grows. For these reasons we focus on a nonlinear network-dependent ordinary differential equation model. To the best of our knowledge, no work has been done on the identification of spread parameters from data for these models. Since these are the main models studied in the controls field, validation of the models is important. 
Many virus spread papers using these models have claimed to use real data to test their models, but no true validation 
of non-trivial network-dependent SIS spread models has been done. Those papers that use real data only build the network structure using real data, but do not have real spread process data over that network. 
In \cite{wan2007network,wan2008designing}, Wan {\em et al.} compare their model to a simulator of SARS, not real data. 
In \cite{chakrabarti2008epidemic}, Chakrabarti {\em et al.} use a router network from the state of Oregon and simulate an artificial spread process over that network. To the best of our knowledge, no work has been done that validates network-dependent SIS models using a set of real spread data. 
Similarly, in \cite{preciado2014optimal}, Preciado {\em et al.} use data from an air transportation network but simulate using arbitrarily chosen healing and infection rates.

We will use two datasets to validate the spread model analyzed in this work. The first dataset is the cholera dataset compiled by John Snow in \cite{snow1855mode}. 
Dr. Snow mapped the deaths caused by cholera in the Soho District of London in 1854 to illustrate that the infection was being spread by contaminated water via a specific pump, the Broad Street pump, and not via the air, as was the belief at the time. This seminal work by Snow has led to the modern day field of epidemiology \cite{bonita2006basic}. While now, partially due to Snow, we understand cholera, how it spreads, and how to mitigate it, it is still a serious problem in poorer parts of the world today. This is highlighted by the current outbreak in Yemen where there have been over 822,000 suspected cases of cholera and over 2,150 cholera-related deaths since the end of April 2017 \cite{whoCholera,pbsCholera}. 

John Snow's original spatial dataset of the cholera epidemic is static and does not contain time series data. Shiode {\em et al.} created spatial time series data in \cite{shiode2015mortality} using various other sources and some statistical methods. However, Shiode {\em et al.} did not perform any dynamic analysis on their dataset, and have not made the dataset publicly available. We use a technique developed in the analysis section herein, combined with several strong but reasonable assumptions, to reproduce time series data, and in so doing, validate the model with the dataset. As far as we know, this is the first attempt to study Snow's cholera dataset from a dynamical systems' perspective to validate models of epidemic processes. 

The second dataset used herein is a record of all the payouts from the United States Department of Agriculture ({USDA}) to farms/farmers for all USDA-sponsored subsidy programs from 2008-2013. For this work we focus on the 2009 Average Crop Revenue
Election (ACRE) Program, which was introduced in that year as an alternative to an existing program, the Direct and Counter-Cyclical Payment (DCP) Program \cite{usdaFact,dcp}. These programs are in place to reduce the risk in the U.S. farming industry, enabling the adoption of new technologies. One of the goals of this paper is to determine whether the adoption of the ACRE program followed a network-dependent discrete-time spread process consistent with the model studied herein. 

A large body of literature in agricultural economics has modeled the adoption and diffusion of agricultural technology, e.g., fertilizer and new seed varieties, (see, e.g., \cite{Sunding2001} for a review of this literature). This literature generally models individuals' decisions to adopt new technologies or the extent of overall adoption, but the spread of information and technology is treated as a ``black box''. Recent work in developing countries has examined whether farmers learn about new technologies from ``information neighbors''. Foster {\em et al.} examine survey data and find that farmers' adoption of high-yielding varieties during the Green Revolution depended on neighbors' experiences \cite{foster1995learning}. Recent evidence from randomized controlled trials shows that farmers learn from their neighbors' experience when the technology is novel or complex \cite{Conley2010a}, but not when adjusting current practices \cite{Duflo2011}. Ghanaian farmers learned from neighbors' experience when switching from traditional crops to pineapple \cite{Conley2010a}, whereas information about optimal fertilizer use for traditional crops in Kenya did not spread among neighbors \cite{Duflo2011}. We take a new approach by using virus spread models to characterize the spread of complex information among U.S. farmers.

A preliminary version of this work has been submitted to the American Control Conference \cite{usda_acc}. The two pieces of work are quite different. Specifically, this paper provides 1) the complete proofs of all the results, 2) additional illustrative simulations, and 3) the validation of the model using the Snow cholera dataset,
which were not included in \cite{usda_acc}.

The paper is organized as follows. In Section~\ref{sec:mod}, the virus spread models are introduced with several remarks that provide insight into how the models are related to each other. 
In Section~\ref{sec:analysis}, we analyze one of the discrete-time spread processes from Section~\ref{sec:mod} that has not been explored in detail. 
In Section~\ref{sec:id}, we present necessary and sufficient conditions for learning, or identifying, the spread process parameters of the same model, from data produced by the models. In so doing, we establish several assumptions that need to be met by the USDA data. 
In Section~\ref{sec:sim}, we validate the results from Section~\ref{sec:id} via simulation. 
In Section~\ref{sec:snow}, we introduce Dr. Snow's seminal cholera dataset from 1854 and use it to validate the spread model. 
In Section~\ref{sec:usda}, we introduce the USDA dataset and the associated subsidy programs, and we learn the homogeneous spread parameters of the ACRE program using data from one part of the country 
and verify the learned parameters by simulating the spread model over the complete contiguous United States and comparing the simulated data with the actual data. We conclude with some discussion of the results and future work in Section~\ref{sec:con}.
\subsection{Notation}

Given a vector function of continuous time $x(t)$, we use $\dot{x}(t)$ to indicate the time-derivative. Given a vector function of discrete time $x^k$, the superscript indicates the time-step of $x$. Given a vector $x \in \mathbb{R}^{n}$, the 2-norm is denoted by $\|x\|$ and the transpose by $x^{\top}$. 
The notation $\0$  denotes the vector whose entries all equal $0$. 
Given two vectors $x_1,x_2\in \mathbb{R}^{n}$, $x_1 > x_2$ indicates each element of $x_1$ is greater than or equal to the corresponding element of $x_2$ and $x_1 \neq x_2$, and $x_1 \gg x_2$ indicates each element of $x_1$ is strictly greater than the corresponding element of $x_2$. 
Given a matrix $A \in \mathbb{R}^{n \times n}$, the maximum eigenvalue 
is $\lambda_1(A)$ (if the spectrum is real), and the largest real-valued part of the eigenvalues of $A$ is denoted by $s_1(A)$ (if the spectrum is possibly complex). 
Also, $a_{ij}$ indicates the $i, j^{th}$ entry of the matrix $A$, and $\| A \|_{F}$ indicates the  Frobenius norm of $A$. 
The notation $diag(\cdot)$ refers to a diagonal matrix with the argument(s) on the diagonal. 

\section{SIS Models}\label{sec:mod}

We introduce two discrete-time SIS models and discuss their relationship. 
For these SIS models, there are two levels of granularity for modeling the system. The state $x_i$ can correspond to a probability of infection of the $i$th agent \cite{van2009virus} or to the percentage of infection of group $i$ \cite{fall2007epidemiological}.
For the identification of the spread process parameters in the USDA dataset in Section \ref{sec:usda}, we employ the latter case.

The first discrete-time model is derived from the continuous-time model 
\begin{equation}\label{eq:cont}
    \dot{x}_{i} = (1-x_{i})\beta_i \sum^{n}_{j=1} a_{ij}x_{j}- \delta_i x_{i},
\end{equation}
where $i$ indicates the $i$th agent or group $i$, $x_i$ is the infection level, $\beta_i>0$ is the infection rate, $\delta_i>0$ is the healing rate, and $a_{ij}\geq 0$, edge weights between the agents/groups. 
Applying Euler's method \cite{atkinson2008introduction} to \eqref{eq:cont} gives 
\begin{equation}\label{eq:dis}
    x_{i}^{k+1} = x_{i}^{k} + h\left((1-x_{i}^{k})\beta_i \sum^{n}_{j=1} a_{ij}x_{j}^{k}- \delta_i x_{i}^{k}\right),
\end{equation}
where $k$ is the time index and $h>0$ is the sampling parameter. 
We can write \eqref{eq:dis} in matrix form 
\begin{equation}\label{eq:disM}
     x^{k+1} = x^k + h((I-X^k)BA-D)x^k,
\end{equation}
where $X^k = diag(x^k)$, $B = diag(\beta_i)$, and $D = diag(\delta_i)$. 
Note that $A$ is the matrix of $a_{ij}$'s and is not necessarily symmetric. 
\begin{remark}
  The model in \eqref{eq:cont} was derived from a mean field approximation of a $2^n$ state Markov chain model \cite{van2009virus}:
\begin{equation}\label{eq:2n}
    \dot{y} = Q y,
\end{equation}
where $Q$ is the transition matrix of the Markov chain 
(the details of the $2^n$ state model are not needed for the discussion here, and hence are not included; for a more detailed discussion, see \cite{pare2017epidemic}). Therefore, \eqref{eq:dis} is an approximation of an approximation.
\end{remark}

An alternative discrete-time model, studied in \cite{ahn2013global}, is
\begin{equation}\label{eq:has}
    x_i^{k+1} = x_i^{k}(1-\delta_i) + (1-x_i^{k})\left( 1 - \prod_{j=1}^n(1 - \beta_i a_{ij}x_j^{k})\right).
\end{equation}
By expanding the model given in \eqref{eq:has}, we obtain
\footnotesize\begin{equation*}
    x_i^{k+1} = x_i^{k} - (1-x_i^{k})\left[ - \beta_i \sum^{n}_{j=1} a_{ij}x_j^{k} + \cdots +  \beta_i^n \prod_{j=1}^n (- a_{ij}x_j^{k})\right] -  \delta_i x_i^{k}.
\end{equation*}\normalsize
\begin{remark}
   If we assume $\beta_i<1 \ \forall i$, the model in \eqref{eq:has} can be approximated by truncating the terms with powers of $\beta_i$ greater than $1$, giving:  
   \begin{equation}\label{eq:app}
       x_i^{k+1} = x_i^{k} + (1-x_i^{k}) \beta_i \sum^{n}_{j=1} a_{ij}x_j^{k} -  \delta_i x_i^{k}.
   \end{equation}
\end{remark}
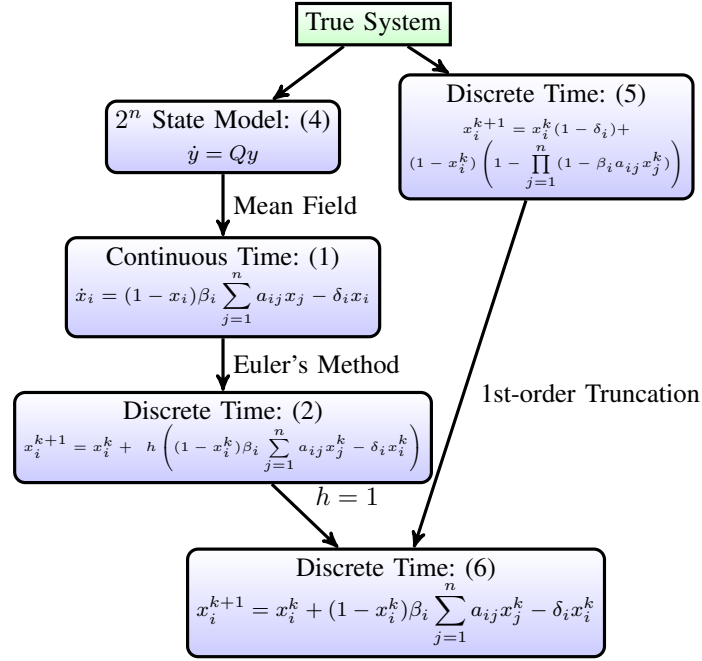
\begin{figure}
\centering
\begin{tikzpicture}[sibling distance=10em,
  every node/.style = {}]]
\begin{scope}[xshift=-7.5cm,yshift=-5cm,very thick,
		node distance=2cm,on grid,>=stealth',
		block/.style={rectangle,rounded corners,draw,align=center,top color=white, bottom color=blue!20},
		comp/.style={rectangle,draw,align=center,top color=white, bottom color=green!20}]
   \node [block] (bot)					{Discrete Time: \eqref{eq:app} \\ \footnotesize $x_i^{k+1} = x_i^{k} + (1-x_i^{k}) \beta_i \displaystyle\sum^{n}_{j=1} a_{ij}x_j^{k} -  \delta_i x_i^{k}$ \normalsize};
   \node [block]	 (mf)	[above=of bot,xshift=-2.3cm,yshift=.2cm]			{Discrete Time: \eqref{eq:dis} \\ \tiny $x_{i}^{k+1} = x_{i}^{k} +  \ \  h\left((1-x_{i}^{k})\beta_i \displaystyle\sum^{n}_{j=1} a_{ij}x_{j}^{k}- \delta_i x_{i}^{k}\right)$ \normalsize} edge [->, "$h=1$"] (bot);
   \node [block] (cont)	[above=of mf]		{Continuous Time: \eqref{eq:cont} \\ \scriptsize $\dot{x}_{i} = (1-x_{i})\beta_i \displaystyle\sum^{n}_{j=1} a_{ij}x_{j}- \delta_i x_{i}$ \normalsize} edge [->, "Euler's Method"] (mf);
   \node [block] (2n)	[above=of cont]		{$2^n$ State Model:  \eqref{eq:2n}\\ \footnotesize $\dot{y} = Qy$ \normalsize} edge [->, "Mean Field" ] (cont);
   \node [block] (dis)	[right=of 2n,xshift=2.3cm]		{Discrete Time: \eqref{eq:has}\\ \tiny $x_i^{k+1} = x_i^{k}(1-\delta_i) + $ \\ \tiny $(1-x_i^{k})\left( 1 - \displaystyle\prod_{j=1}^n(1 - \beta_i a_{ij}x_j^{k})\right)$ \normalsize} edge [->, "1st-order Truncation" ] (bot);
   \node [comp]  (sys)  [above=of 2n,xshift=2cm,yshift=-.5cm] {True System} edge [->] (dis) edge [->](2n);
\end{scope}
\end{tikzpicture}
\caption{A graphical illustration of the discussion in Section \ref{sec:mod} and the point in Observation \ref{obs:app}, showing how the two discrete-time spread models are related. The first modeling layer shows the $2^n$ state models. The arrows indicate different approximations taken.}
\label{fig:mod}
\end{figure}
\noindent The preceding discussion leads us to the following observation.
\begin{observation}\label{obs:app}
   The approximation given by \eqref{eq:app} and the discrete approximation of the mean field approximation of the continuous $2^n$ state Markov model in \eqref{eq:dis} are equivalent, given $h=1$.
\end{observation}
\noindent The relationships between the models introduced in this section are depicted in Figure~\ref{fig:mod}. 
The first layer of modeling is the most detailed, where the left side is the model given by \eqref{eq:2n} and the right side is that given by \eqref{eq:has}.



\section{Analysis}
\label{sec:analysis}

In this section, a different 
version of the model in \eqref{eq:dis} will be analyzed, as follows: 
\begin{equation}\label{eq:disG}
     x^{k+1} = x^k + h((I-X^k)B-D)x^k,
\end{equation}
where $[B]_{ij} = \beta_{ij}$, capturing the infection rate and nearest-neighbor graph structure in one. Note $\beta_{ij}$ could be factored into $\beta_i  a_{ij}$ as in \eqref{eq:dis}.
\begin{assumption}
For all $i\in[n]$, we have $x^{0}_i\in[0,1]$.
\label{x0}
\end{assumption}

\begin{assumption}
For all $i\in[n]$, we have $\delta_i\geq0$ and, for all $j\in[n]$, $\beta_{ij} \geq 0$.
\label{pos}
\end{assumption}

\begin{assumption}
For all $i\in[n]$, we have $h\delta_i\leq 1$ and $h\sum_{j=1}^n \beta_{ij}\leq 1$.
\label{01}
\end{assumption}


\begin{lemma}
For the system in \eqref{eq:disG}, under the conditions of Assumptions  \ref{x0}, \ref{pos}, and \ref{01}, $x^{k}_i\in[0,1]$ for all $i\in[n]$ and $k\ge 0$.
\label{lem:box}
\end{lemma}
\begin{proof}
Suppose that at some time $k$, $x^{k}_i\in[0,1]$ for all $i\in[n]$.
Consider an index $i\in[n]$.
Rearranging \eqref{eq:dis}, 
$$x_{i}^{k+1} = x_{i}^{k}(1- h\delta_i) + (1-x_{i}^{k})\left(h \sum^{n}_{j=1} \beta_{ij}x_{j}^{k}\right),$$ 
we see that $x^{k+1}_i$ is a convex combination of $(1-h\delta_i)$ and $h \sum^{n}_{j=1} \beta_{ij}x_{j}^{k}$. Since, by Assumptions \ref{pos} and \ref{01}, $h\delta_i, h \sum^{n}_{j=1} \beta_{ij}x_{j}^{k} \in [0,1]$, we have $x_{i}^{k+1} \in [0,1]$.



Further, by Assumption \ref{x0}, $x^{0}_i\in[0,1]$ for all $i\in[n]$, thus
it follows that $x^{k}_i\in[0,1]$ for all $i\in[n]$ and $k\ge 0$.
\end{proof}
Lemma \ref{lem:box} implies that the set $[0,1]^n$ 
is positively invariant with respect to the system defined by \eqref{eq:disG}. Since $x_i$ 
denotes the probability of infection of individual $i$, or the fraction  of group $i$ infected, 
and $1-x_i$ denotes probability of individual $i$ being healthy, or the fraction of group $i$ that is healthy, it is natural to assume that their initial values are in the interval $[0,1]$, 
since otherwise the values will lack any physical meaning for the epidemic model considered here. 
Therefore, 
we focus on the analysis of \eqref{eq:disG} only on the domain $[0,1]^n$.

We need an assumption to ensure \textit{non-trivial} 
virus spread.
\begin{assumption}\label{not0}
 We have $h\neq 0$ and $\exists i\neq j$ s.t. $\beta_{ij}>0$.
\end{assumption}
Note that we do not assume the healing rates to be nonzero. This allows for the possibility of SI (susceptible-infected) models \cite{zhou2006behaviors}. 
\begin{definition}
Consider an autonomous system
\begin{equation}
     x^{k+1}  = f(x^k), \label{def}
\end{equation}
where $f: \scr{X}\rightarrow\R^n$ is a locally Lipschitz map from a domain $\scr{X}\subset\R^n$ into $\R^n$. 
Let $z$ be an equilibrium of \eqref{def} and $\scr{E}\subset\scr{X}$ be a domain containing $z$.
If the equilibrium $z$ is asymptotically stable such that for any $x^0\in\scr{E}$ we have
$\displaystyle\lim_{k\rightarrow\infty}x^k = z$, then $\scr{E}$ is said to be a domain of attraction for~$z$.
\end{definition}

\begin{prop}
Let $z$ be an equilibrium of \eqref{def} and $\scr{E}\subset \scr{X}$ be a domain
containing $z$. Let $V:\scr{E}\rightarrow\R$ be a continuously differentiable function
such that $V(z)=z$, $V(x)>0$ for all $x$ in $\scr{E}\setminus \{z\}$, 
and $\Delta V^k := V(x^{k+1}) - V(x^k)<0$ for all $x^k$ in $\scr{E}\setminus \{z\}$. If $\scr{E}$ is a positively invariant set,
then the equilibrium $z$ is asymptotically stable with a domain of attraction $\scr{E}$.
\label{prop:lya}
\end{prop}
\noindent This proposition is a direct consequence of Lyapunov's stability theorem for discrete-time systems, which can be found in \cite{vidyasagar2002nonlinear},
and the definition of domain of attraction.

Finally, we need an assumption on the structure of the $B$ matrix. A square matrix is called {\em irreducible}  if it cannot be permuted to a block upper triangular matrix.
\begin{assumption}
The matrix $B$ is irreducible.
\label{connect}
\end{assumption}
\noindent Note that this assumption is equivalent to the underlying graph being strongly connected. 

\begin{theorem} \label{thm:0global}
Suppose that Assumptions \ref{x0}-\ref{connect} hold for \eqref{eq:disG}. If $s_1(I-hD+hB)\leq 1$, then the healthy state is asymptotically stable with  domain of attraction $[0,1]^n$.
\end{theorem}

To prove the theorem, we need the following lemmas. 
\begin{lemma} {\em \cite{rantzer2011positive}}
Suppose that M is an irreducible nonnegative matrix such that $s_1(M) < 1$. Then, there exists a positive diagonal matrix $P$ such that $M^{\top} P M - P$ is negative definite. \label{lem:ndef}
\end{lemma}

\begin{lemma} 
Suppose that M is an irreducible nonnegative matrix such that $s_1(M) = 1$. Then, there exists a positive diagonal matrix $P$ such that $M^{\top} P M - P$ is negative semi-definite. \label{lem:sndef}
\end{lemma}
\begin{proof}
From the Perron Frobenius Theorem for irreducible nonnegative matrices, there exists $v \gg 0$ such that $Mv = v$. Since $M^{\top}$ is also irreducible and nonnegative, there exists $u \gg 0$ such that $M^{\top} u = u$. Let $P$ be a diagonal matrix whose $i$th diagonal entry is equal to $u_i/v_i$, which gives $Pv = u$. Therefore,
\begin{align*}
    (M^{\top} P M - P)v = M^{\top} P v - Pv = M^{\top} u - u = 0.
\end{align*}
Then by Lemma 2.3 in \cite{varga}, $s_1(M^{\top} P M - P) = 0$.
\end{proof}

{\em Proof of Theorem \ref{thm:0global}.}
To simplify notation, let $M = I+hB-hD$ and $\hat{M} = I + h((I-X^k)B-D)$. By Assumptions \ref{pos}-\ref{connect}, $M$ is an irreducible nonnegative matrix. First we evaluate the case where $s_1(I-hD+hB)< 1$. Therefore, by Lemma \ref{lem:ndef}, there exists a positive diagonal matrix $P_1$ such that $M^{\top} P_1 M - P_1$ is negative definite. Consider the Lyapunov function $V_1(x^k) = (x^k)^{\top} P_1 x^k$. 
Using \eqref{eq:disG} with $x^k\neq 0$, gives
\begin{align}
  \Delta V_1^k &= (x^{k})^{\top}\hat{M}^{\top} P_1 (x^{k})^{\top}\hat{M} x^{k} - (x^k)^{\top}P_1 x^k \nonumber \\
  &= (x^{k})^{\top}(M^{\top} P_1 M - P_1) x^{k} - 2h(x^{k})^{\top}B^{\top}X^KP_1 M x^k  + h^2(x^{k})^{\top}B^{\top}X^KP_1 X^K B x^k \nonumber \\ 
  &< h^2(x^{k})^{\top} B^{\top} X^KP_1 X^K B x^k  - 2h(x^{k})^{\top}B^{\top}X^KP_1 M x^k \label{eq:pos} \\ 
  &= h^2(x^{k})^{\top}B^{\top}X^KP_1 X^K B x^k    - 2h^2(x^{k})^{\top}B^{\top}X^KP_1 B^{\top} x^k  - 2h(x^{k})^{\top}B^{\top}X^KP_1 (I-hD) x^k  \nonumber \\ 
  &\leq h^2((x^{k})^{\top}B^{\top}X^KP_1 X^K B x^k   - 2(x^{k})^{\top}B^{\top}X^KP_1 B^{\top} x^k) \label{eq:hD} \\
  &\leq -h^2(x^{k})^{\top}B^{\top}X^KP_1(I - X^K) B x^k \nonumber \\
  &\leq 0, \label{eq:0}
\end{align}
where \eqref{eq:pos} holds by Lemma \ref{lem:ndef}, \eqref{eq:hD} holds by Assumptions \ref{pos} and \ref{01}, and \eqref{eq:0} holds by Lemma \ref{lem:box}. 
Therefore, by Proposition~\ref{prop:lya}, the system converges asymptotically to the healthy state for this case.

For the case where $s_1(I-hD+hB)= 1$, we have, by Lemma \ref{lem:sndef}, that there exists a positive diagonal matrix $P_2$ such that $M^{\top} P_2 M - P_2$ is negative semi-definite. Consider the Lyapunov function $V_2(x^k) = (x^k)^{\top} P_2 x^k$. Using \eqref{eq:disG} with $x^k\neq 0$, gives
\begin{align}
  \Delta V_2^k &= (x^{k})^{\top}\hat{M}^{\top} P_2 (x^{k})^{\top}\hat{M} x^{k} - (x^k)^{\top} P_2 x^k \nonumber \\
  &= (x^{k})^{\top}(M^{\top} P_2 M - P_2) x^{k} - 2h(x^{k})^{\top}B^{\top}X^KP_2 M x^k   + h^2(x^{k})^{\top} B^{\top} X^KP_2 X^K B x^k \nonumber \\ 
  &< h^2(x^{k})^{\top}B^{\top}X^KP_2 X^K B x^k - 2h(x^{k})^{\top}B^{\top}X^KP_2 M x^k \nonumber \\ 
  &= h^2(x^{k})^{\top}B^{\top}X^KP_2 X^K B x^k   - h(x^{k})^{\top}B^{\top}X^KP_2 M x^k \nonumber \\ 
  &  \ \ \ \ \ \ -h^2(x^{k})^{\top}B^{\top}X^KP_2 B x^k  - h(x^{k})^{\top}B^{\top}X^KP_1 (I-hD) x^k  \nonumber \\ 
  &\leq h^2(x^{k})^{\top}B^{\top}X^KP_2 X^K B x^k   - h(x^{k})^{\top}B^{\top}X^KP_2 M x^k  -h^2(x^{k})^{\top}B^{\top}X^KP_2 B x^k  \nonumber \\ 
  &\leq h^2(x^{k})^{\top}B^{\top}X^KP_2 (I - X^K) B x^k   - h(x^{k})^{\top}B^{\top}X^KP_2 M x^k  \nonumber  \\
  &\leq - h(x^{k})^{\top}B^{\top}X^KP_2 M x^k \nonumber \\
  &\leq 0. \nonumber
\end{align}
Clearly if $x^k = \0$, then $- h(x^{k})^{\top}B^{\top}X^KP_2 M x^k = 0$. Since, by Assumptions \ref{pos} and \ref{not0}, $B,P_2,M$ are nonzero, nonnegative matrices, if $- h(x^{k})^{\top}B^{\top}X^KP_2 M x^k = 0$, then $x^k = \0$. Therefore, by Proposition \ref{prop:lya}, the healthy state is asymptotically stable with  domain of attraction $[0,1]^n$.
\hfill
$\qed$

\begin{prop}\label{prop:endemic}
   Suppose that Assumptions \ref{x0}-\ref{connect} hold. If $s_1(I-hD+hB)> 1$, then \eqref{eq:disG} has two equilibria, $\0$ and $x^*$, and $x^*\gg \0$.
\end{prop}
\begin{proof}
  Clearly $\0$ is always an equilibrium of \eqref{eq:disG}.
  
  Note that
  $$ s_1(I-hD+hB) = 1 + h(s_1(-D+B)).$$
  Therefore, 
  $$ s_1(I-hD+hB)> 1 \Longleftrightarrow h(s_1(-D+B))> 0.$$
  This condition is the same as the condition of Proposition 3 in \cite{liu2016onthe,arxiv}, and the proof follows similarly, showing that there exists $x^*\gg \0$ such that 
  $$h((-D+B)-X^*B)x^*=\0.$$
  Therefore, $\0$ and $x^*$ are equilibria of \eqref{eq:disG}.
\end{proof}

From Theorem \ref{thm:0global} and Proposition \ref{prop:endemic}, we have the following result.

\begin{theorem}
Under Assumptions \ref{x0}-\ref{connect}, the healthy state is the unique equilibrium of \eqref{eq:disG} if and only if
 $s_1(I-hD+hB)\leq 1$.
\label{thm:eq0}
\end{theorem}

In \cite{ahn2013global}, a counterexample is provided to show that the nontrivial equilibrium of \eqref{eq:has} is unstable. However, this example does not hold for the models in \eqref{eq:dis} and \eqref{eq:disG} because it does not meet Assumption \ref{01}. 
Consequently the state of the system does not stay in the domain of interest, $[0,1]^n$.


\begin{remark}
If the system has homogeneous spread parameters, the condition in Theorems \ref{thm:0global}-\ref{thm:eq0} reduces to $s_1(A) \leq \frac{\delta}{\beta}$.
\label{rem:thres}
\end{remark}

\section{Learning Spread Parameters}\label{sec:id}

In this section, we clearly lay out the assumptions and the identification techniques for several versions of the model in \eqref{eq:dis}, introduced in Section \ref{sec:mod}. We assume that the underlying graph structure $A$ is known and that we have full-state measurement with no noise on the measurements, which we admit are strong assumptions. However, for the application considered here these assumptions are well-founded  because we aggregate the data by county and the adjacency of counties is known, i.e., the graph structure is known, and  any farmer that received a subsidy payout is in the dataset, i.e., there are no hidden, unmeasured states. 


We present several results on learning the spread parameters of the model in \eqref{eq:dis} from data. 
\begin{theorem}\label{thm:idhomo}
Consider the model in \eqref{eq:dis} under Assumptions \ref{x0}-\ref{connect} with 
homogeneous virus spread, that is, $\beta$ and $\delta$ are the same for all $n$ agents, 
with $n>1$. 
Assume that $A$, $x^0, \dots, x^T$,  and $h$ are known. Then, the spread parameters can be learned uniquely if and only if $T>0$, and 
there exist $i,j \in [n]$ and $ l_1, l_2\in[T-1] \cup \{0\}$ such that
\begin{align}
x_j^{l_2} g_i(x^{l_1})
\label{eq:inv1}
- 
x_i^{l_1} g_j(x^{l_2}) \neq 0,
\end{align}
where $g(x^{k}):=h(I-X^{k})A x^{k}$.
\end{theorem}
\begin{proof}
Since $A$, $x^0, \dots, x^{T-1}$,  and $h$ are known, using the notation in \eqref{eq:disM} we can construct the matrix 
\begin{equation*}
     \Phi = \begin{bmatrix} h(I-X^0)A x^0 & -hx^0\\ \vdots & \vdots \\ h(I-X^{T-1})A x^{T-1} & -hx^{T-1} \end{bmatrix}.
\end{equation*}
Therefore, since $x^T$ is also known, we can rewrite \eqref{eq:dis} as 
\begin{equation}\label{eq:id1}
    \begin{bmatrix} x^1-x^0\\ \vdots \\ x^T-x^{T-1} \end{bmatrix} =  \Phi \begin{bmatrix} \beta \\ \delta \end{bmatrix}.
\end{equation}
Since $n>1$, $\Phi$ has at least two rows. 
By the assumption that there exist $i,j \in [n]$ and  $ l_1, l_2\in[T-1] \cup \{0\}$ such that 
\eqref{eq:inv1} holds, 
$\Phi$ has 
column rank equal to two, with two unknowns. 
Therefore there exists a unique solution to \eqref{eq:id1} using the inverse or pseudoinverse.

If there do not exist $i,j \in [n]$ and  $ l_1, l_2\in[T-1] \cup \{0\}$ such that \eqref{eq:inv1} holds, then $\Phi$ has a nontrivial nullspace. Therefore \eqref{eq:id2} does not have a unique solution. 
\end{proof}
Now we present two corollaries where the ratio of the spread parameters, $\delta/\beta$, can be recovered. 
\begin{corollary}\label{cor:h}
 Consider the model in \eqref{eq:dis} under Assumptions \ref{x0}-\ref{connect} 
 with homogeneous virus spread 
with $n>1$.  
    Assume that $A$ and $x^0, \dots, x^T$ are known. Then, the ratio of the spread parameters can be learned uniquely if and only if $T>0$ and
there exist $i,j \in [n]$ and $ l_1, l_2\in[T-1] \cup \{0\}$ such that
\begin{align*}
x_j^{l_2} g_i(x^{l_1})
- 
x_i^{l_1} g_j(x^{l_2})\neq 0.
\end{align*}
\end{corollary}
\begin{proof}
Since $h$ factors out of the right hand side of \eqref{eq:id1} and is nonzero by Assumption \ref{not0}, even if $h$ is not known, a scaled version of the pair $\beta$ and $\delta$ can be recovered exactly. Therefore, the proportion of the two parameters can be found.
\end{proof}
\begin{corollary}\label{cor:xs}
Considering the model in \eqref{eq:dis} under Assumptions \ref{x0}-\ref{connect}, if the endemic state, $x^*\gg 0$, exists, then 
\begin{equation}
\frac{\delta_i}{\beta_i} = \frac{(1-x_{i}^*)}{x_{i}^*} \sum^{n}_{j=1} a_{ij}x_{j}^*.   
\end{equation}
\end{corollary}
\begin{proof}
This follows from solving \eqref{eq:dis}, using $x_i^*>0$ in place of $x_{i}^{k+1}$ and $x_{i}^{k}$.
\end{proof}
These corollaries illustrate that under certain conditions, while the exact behavior of the system may not be recoverable, 
the limiting behavior of the system may be determined, by employing Theorems \ref{thm:0global}-\ref{thm:eq0} with Remark \ref{rem:thres}.

If the assumption is made that the underlying spread process is heterogeneous, a similar result to Theorem \ref{thm:idhomo} can be concluded. 
\begin{theorem}
Consider the model in \eqref{eq:dis} under Assumptions \ref{x0}-\ref{connect} with $n>1$. 
Assume that $A$, $x^0, \dots, x^{T-1}$, $x_i^{T}$,  and $h$ are known. Then, the spread parameters of node $i$ can be learned uniquely if and only if $T>1$, and 
there exist $ l_1, l_2\in[T-1] \cup \{0\}$ such that
\begin{align}\label{eq:inv2}
x_i^{l_2} (1-x_{i}^{l_1})\sum^{n}_{j=1} a_{ij}x_{j}^{l_1} 
-  x_i^{l_1} &(1-x_{i}^{l_2})\sum^{n}_{j=1} a_{ij}x_{j}^{l_2}
\neq 0.
\end{align}
\end{theorem}
\begin{proof}
Since $A$, $x^0, \dots, x^{T-1}$,  and $h$ are known, 
we can construct the matrix 
\begin{equation*}
     \Phi_i = \begin{bmatrix} \displaystyle h(1-x_{i}^{0}) \sum^{n}_{j=1} a_{ij}x_{j}^{0}  & \displaystyle -hx_{i}^0\\ \vdots & \vdots \\ \displaystyle h(1-x_{i}^{T-1}) \sum^{n}_{j=1} a_{ij}x_{j}^{T-1}  & \displaystyle -hx_{i}^{T-1} \end{bmatrix}.
\end{equation*}
Then, since $x_i^{T}$ is known, we have 
\begin{equation}\label{eq:id2}
    \begin{bmatrix} \displaystyle x_{i}^1-x_{i}^0\\ \vdots \\ \displaystyle x_{i}^T-x_{i}^{T-1} \end{bmatrix} =  \Phi_i \begin{bmatrix} \beta_{i} \\ \delta_{i} \end{bmatrix}.
\end{equation}
Since $T>1$, $\Phi_i$ has at least two rows. 
By the assumption that there exist $ l_1, l_2\in[T-1] \cup \{0\}$ such that 
\eqref{eq:inv2} holds, 
$\Phi_i$ has column rank equal to two, with two unknowns. 
%
Therefore there exists a unique solution to \eqref{eq:id2} using the inverse or pseudoinverse.

If there do not exist $ l_1, l_2\in[T-1] \cup \{0\}$ such that \eqref{eq:inv2} holds, then $\Phi_i$ has a nontrivial nullspace. Therefore, in that case, \eqref{eq:id2} does not have a unique solution. 
\end{proof}
\normalsize

\noindent Identifying heterogeneous spread parameters, however interesting, will not help identify the spread in other areas. Therefore, a homogeneous system should be more informative for some applications. 
For the Snow dataset in Section \ref{sec:snow}, we will employ the heterogeneous approach, using Corollary~\ref{cor:xs} and assuming $\beta_i = 1$ for all $i$.  
We will employ homogeneous 
formulation on the USDA dataset in Section~\ref{sec:usda}. 





\section{Simulations}\label{sec:sim}

In this section, we present a  simulation that implements the results of the previous section. 
While the data used in this section is generated in Matlab, the insights gained from the exercises here contribute towards our approach using the USDA dataset in Section~\ref{sec:usda}. 



\begin{figure}
    \centering
    \begin{subfigure}[b]{.493\columnwidth}
      \includegraphics[width=\columnwidth]{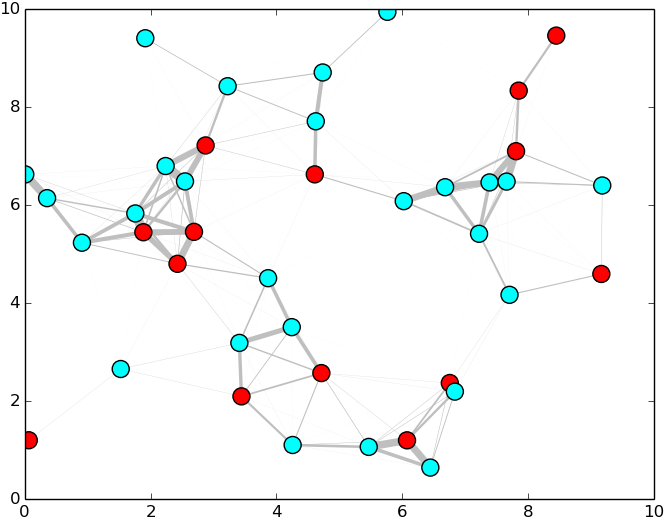}
      \caption{The system at time zero.}
      \label{fig:dis0}
    \end{subfigure}
    \hfill
    \begin{subfigure}[b]{.493\columnwidth}
      \includegraphics[width=\columnwidth]{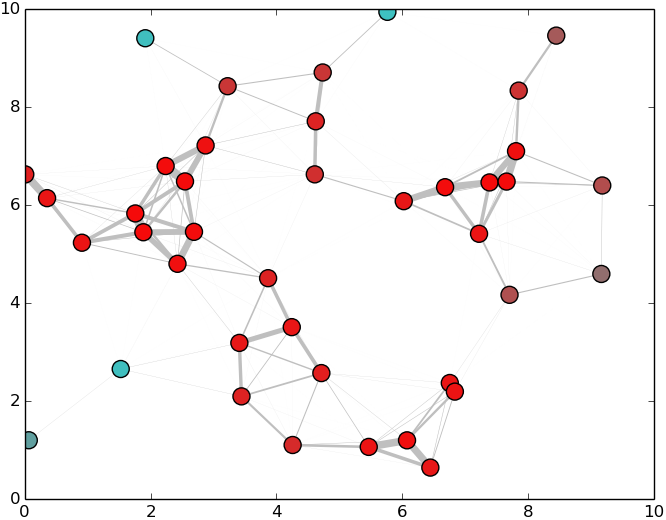}
      \caption{The system at time 100.}
      \label{fig:dis100}
    \end{subfigure}
    \caption{This virus system follows \eqref{eq:dis} with $\beta =1$, $\delta = .1$, $h=.1$, and $A$ depicted by the edges.  
    Teal indicates healthy or susceptible, while red indicates infected. 
For a video of this simulation please see \href{http://youtu.be/JhU1mEvlV-g}{youtu.be/JhU1mEvlV-g}.}
\label{fig:dis}
\end{figure}
Consider a system with 40 agents, with a random set of initially infected agents, with $\beta =1$, $\delta = .1$, $h=.1$ and the weighting matrix $A$ determined by the agents' relative positions given by $z_i$, that is, for radius $r=2$ and $i\neq j$,
\begin{equation}\label{eq:at}
a_{ij}(t) = \begin{cases}
    e^{-\|z_i - z_j\|^2}, & \text{if } \|z_i - z_j\| < r \\
    0,              & \text{otherwise}.
\end{cases}
\end{equation}
See Figure \ref{fig:dis} for plots of the initial and final conditions. Assuming that the correct value for $h$ and the $A$ matrix are known, using \eqref{eq:id1} exactly recovers $\beta$ and $\delta$. 
If only two time-steps are used, the exact spread parameters can be recovered, consistent with Theorem \ref{thm:idhomo}. 
Using \eqref{eq:id1} with an incorrect $h$ value to recover $\beta$ and $\delta$ gives incorrect values for $\beta$ and $\delta$, but results in the right proportion between the two, consistent with Corollary \ref{cor:h}. 
If the system is at the endemic state, the proportion between the spread parameters can be solved exactly using Corollary~\ref{cor:xs}.

\section{Validation: Snow Dataset} \label{sec:snow}

Now we employ the seminal cholera dataset collected by John Snow \cite{snow1855mode} for validation of the model in \eqref{eq:dis}.

\begin{figure}
\centering
    \includegraphics[width = .75\columnwidth]{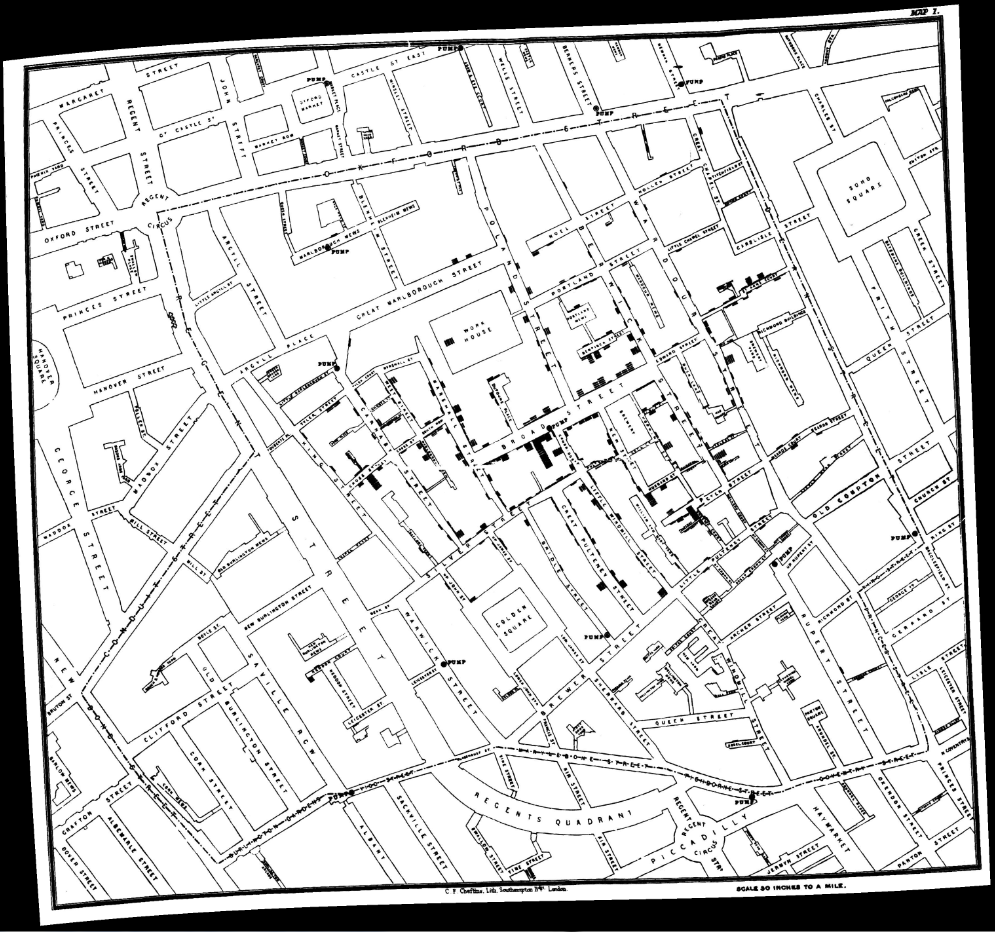}
    \caption{This is the map of cholera spread in London in 1854 compiled by John Snow \cite{snow1855mode}.}\label{fig:snow}
\end{figure}

\begin{figure}
\centering
    \includegraphics[width = .75\columnwidth]{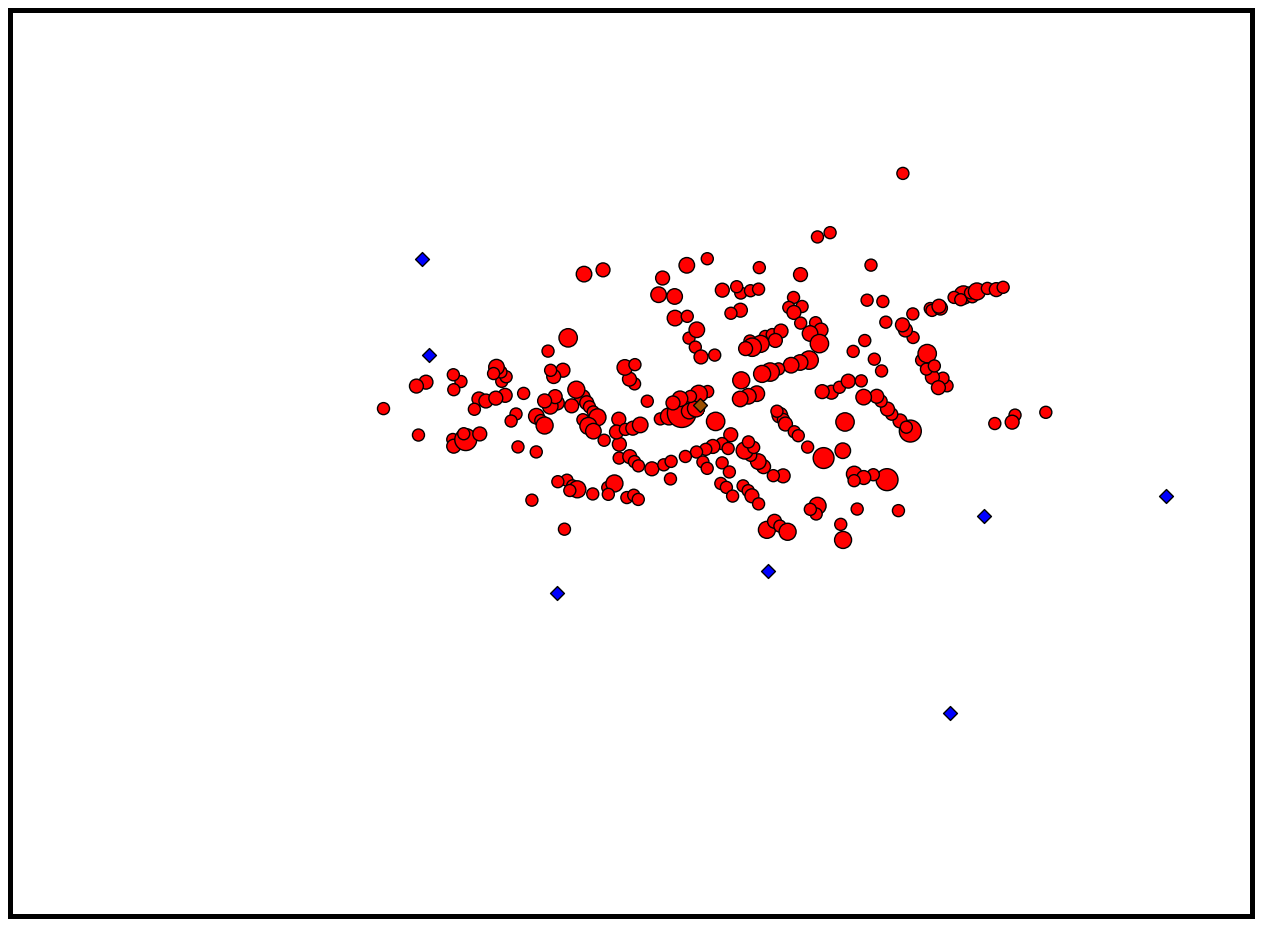}
    \caption{This is a digitization of the data from Figure \ref{fig:snow}. The blue diamonds indicate healthy water pumps, the brown diamond indicates the contaminated Broad Street pump, and the red dots indicate deaths with the diameter scaled by the number of deaths.}\label{fig:snowdata}
\end{figure}

\subsection{Snow Dataset} \label{subsec:snow}

Snow depicted the total number of deaths caused by cholera in the Soho District of London in 1854 on a map of the area. In Figure \ref{fig:snow}, the original map is shown, where each small rectangle corresponds to one death at that address. Snow  created this map to illustrate to officials that the cholera epidemic was being spread by infected water from the Broad Street pump, and not through  the air, as was the common belief of those times. We have plotted this data in Figure \ref{fig:snowdata}, with diamonds indicating a water pump and red dots indicating deaths. Snow also documented the cumulative deaths per day in Table I of \cite{snow1855mode}, plotted in Figure \ref{fig:deaths}. The time of deaths for each address is not recorded. Note that the total cumulative deaths is 616, but the total number of deaths on the map are 489. Therefore, there is a discrepancy of 127 deaths, whose household addresses are not included in the map. 
For validation of the model in \eqref{eq:dis} we use the number of deaths as the metric for the disease spread.

\begin{figure}
    \includegraphics[width=\columnwidth]{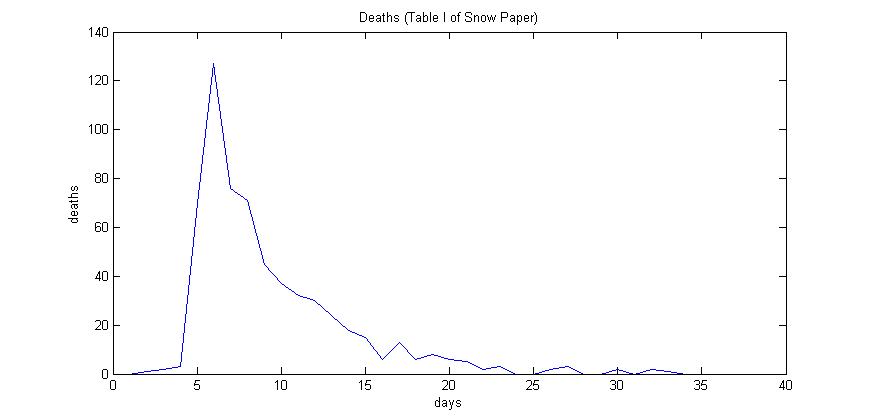}
    \caption{Deaths per day in the Soho District of London in 1854 compiled by John Snow \cite{snow1855mode}. }
    \label{fig:deaths}
\end{figure}

\subsection{Spread Validation}

For the validation, each household with a death recorded by Snow in the map in Figure \ref{fig:snow} corresponds to a  node in the model. The last node in the model corresponds to the contaminated pump, the one on Broad Street, and we do not include the healthy water pumps in the model. We realize that ignoring the households with no recorded deaths and ignoring the healthy pumps are nontrivial assumptions. However, as was noted by Snow, many residents fled the city once they became aware of the outbreak \cite{snow1855mode}. For the households that did not flee, we assume they either had such a high healing rate that their inclusion would have been trivial and/or that these households exclusively drank from another pump and did not closely associate with neighbors who did drink from the Broad Street pump. Despite these (and subsequent) relatively strong assumptions the validation results are quite promising. 

The state of the system, $x^k$, is the percentage of total deaths in each household up to time $k$. The epidemic equilibrium of the system, which we call $x^*$, was calculated from the data in Figure \ref{fig:snowdata}, for the first attempts, by dividing the total number of deaths in each household by 20, and therefore assuming that each household has 20 members. 
This number was chosen because the maximum number of deaths was 15. 
For the third attempt we approximated the household sizes using Figure 1 in \cite{shiode2015mortality}. 
The last element of $x^*$, corresponding to the pump was set to $\frac{19}{20}$. 
Then, assuming $\beta_i = 1$ for all $i$, we employed Corollary \ref{cor:xs} to calculate the $\delta_i$ values. 
Recall the Broad Street pump corresponds to the last agent in the model (agent $n$). For the initial condition in the simulations, we began with the Broad Street pump infected and all the households healthy:
\begin{equation}
x^0 = \begin{bmatrix}
0 & \hdots & 0 & 1   
\end{bmatrix}^{\top}.
\end{equation}
This initial condition is shown in Figure \ref{fig:snowx0}, where the contaminated pump is depicted as a brown diamond. 
As a consequence of these assumptions, our tuning parameter for adjusting the learned $\delta_i$ parameters, and consequently the spread behavior, was the connectivity matrix $A$. 

\begin{figure}
\centering
\begin{subfigure}{\linewidth}
\centering
    \includegraphics[width = .75\columnwidth]{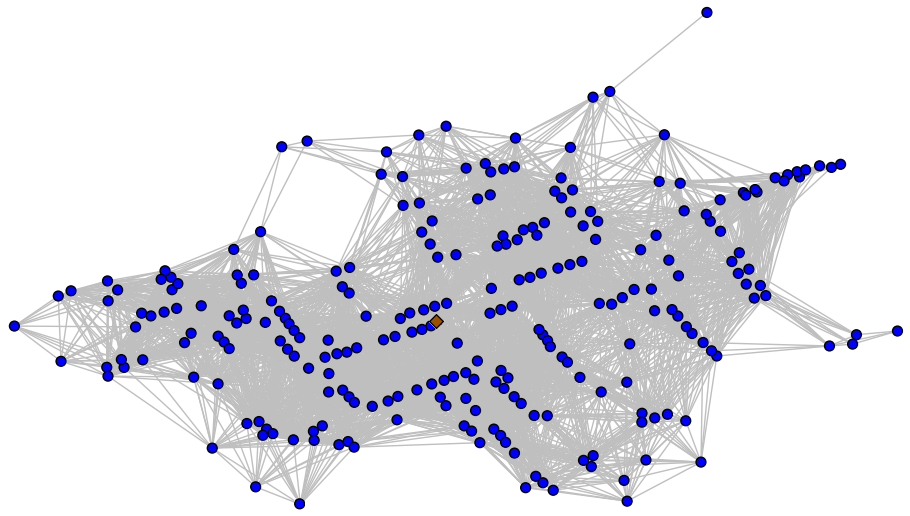}
    \caption{Initial condition of simulations: blue circles indicate healthy households and the brown diamond indicates the infected pump.}
    \label{fig:snowx0}
    \end{subfigure}
    \begin{subfigure}{\linewidth}
    \centering
    \includegraphics[width = .75\columnwidth]{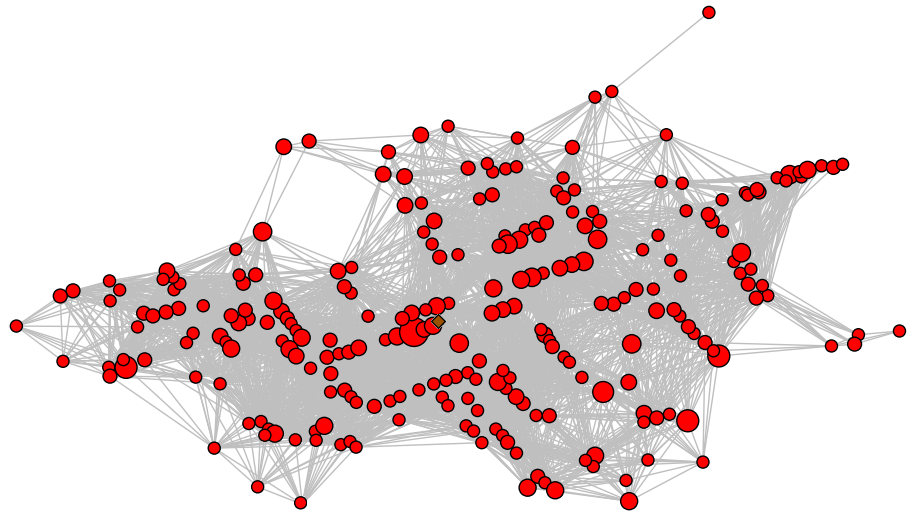}
    \caption{Final state, the epidemic equilibrium, or endemic state of the system: the diameters of the nodes scale with the number of deaths in that household.}
    \end{subfigure}
    \caption{These two plots show the initial and final states of the simulations. The connectivity here corresponds to $A^1$ from~\eqref{eq:Abad}. }\label{fig:vid}
\end{figure}

\begin{figure}
    \includegraphics[width=\columnwidth]{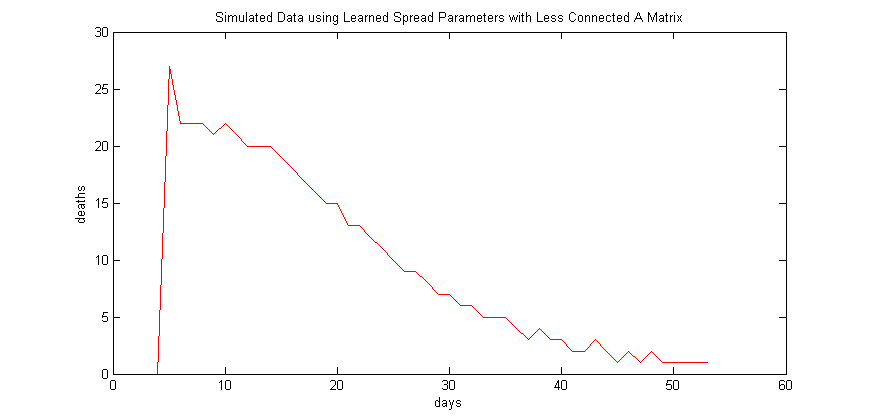}
    \caption{Simulated data using the learned parameters from the data in Figure \ref{fig:snowdata}, employing Corollary \ref{cor:xs} and $A^1$ from~\eqref{eq:Abad}.}
    \label{fig:deaths_sim_bad}
\end{figure}

For the first attempt, we designed $A^1$ such that
\begin{equation}
a^1_{ij} = \begin{cases}
    1, & \text{ if }\|z_i - z_j\| < r,\\
    1, &\text{ if } i=j\text{,} \\
    0, & \text{ otherwise,}
    \end{cases}
\label{eq:Abad}
\end{equation}
where $z_i$ is the location of household $i$ and $r$ was chosen such that the graph was connected. The graph imposed by the $A^1$ matrix is shown in Figure \ref{fig:vid}. Using the $\delta_i$ parameters derived from $A^1$, we simulated the system, using \eqref{eq:dis}. 
To meet the constraints of Assumption \ref{01}, we had to set $h = \frac{1}{175}$. 
This simulation resulted in the distribution of deaths shown in Figure \ref{fig:deaths_sim_bad}; this plot was created by multiplying the state of the system, percentage of deaths in each household up to that point, by the household sizes, assumed to be 20, rounding to the nearest integer, taking the difference between the states of each time step (since the state represents cumulative number of deaths up to that point),
and then summing up every three time series points (due to the small $h$ value), therefore assuming that each time series point corresponds to a third of a day. Note that the shape is very different than the dataset, shown in Figure \ref{fig:deaths}. 

\begin{figure}
    \includegraphics[width=\columnwidth]{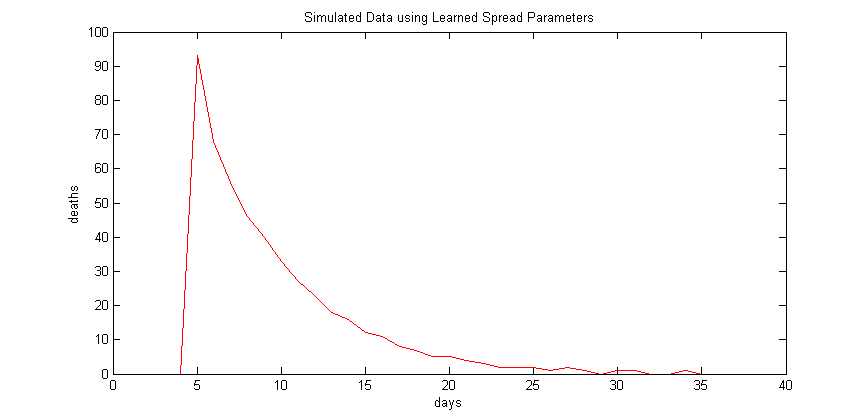}
    \caption{Simulated data using the learned parameters from the data in Figure \ref{fig:snowdata}, employing Corollary \ref{cor:xs} and $A^2$ from~\eqref{eq:Apump}.}
    \label{fig:deaths_sim}
\end{figure}

For the second attempt, since it is well known (now) that cholera spreads primarily through contaminated water, and we know that the Broad Street pump was the source of this epidemic, we allowed the pump to affect everyone. This was done by setting
\begin{equation}
A^2 = [ \ A^1(1:n,1:n-1)  \ v \ ],
\label{eq:Apump}
\end{equation}
where $v= \1 \in \R^n$ and the notation $A^1(1:n,1:n-1)$ indicates all of the $A^1$ matrix except the last column. 
Using the $\delta_i$ parameters derived from $A^2$, we simulated the system, again setting $h = \frac{1}{175}$. The resulting distribution of deaths is shown in Figure \ref{fig:deaths_sim} (created similarly to Figure \ref{fig:deaths_sim_bad}). 
Note that the shape is very similar to the original dataset from \cite{snow1855mode}, shown in Figure \ref{fig:deaths}, capturing the behavior of the true epidemic. 

Plotting the distributions from Figures \ref{fig:deaths} and \ref{fig:deaths_sim} on the same plot for comparison in Figure \ref{fig:comp} shows that they are not identical. This difference results from the fact, noted in Section \ref{subsec:snow}, that the total number of deaths in the map (Figure \ref{fig:snow}), 
used to derive $x^*$ and consequently the spread parameters and the simulation, is 489, and the total number of deaths in Table I of \cite{snow1855mode}, 
used to create the distribution of deaths over days in Figure~\ref{fig:deaths}, is 616. Therefore, the lack address information of the additional 127 deaths results in this inaccuracy. 
However, the largest discrepancy occurs near the peak of the epidemic, when people were arriving at hospitals too sick to provide their addresses \cite{snow1855mode}. Nevertheless, the results are very promising showing that the model in \eqref{eq:dis} captures the behavior of the cholera epidemic from John Snow's 1854 dataset quite well. 

\begin{figure}
    \includegraphics[width=\columnwidth]{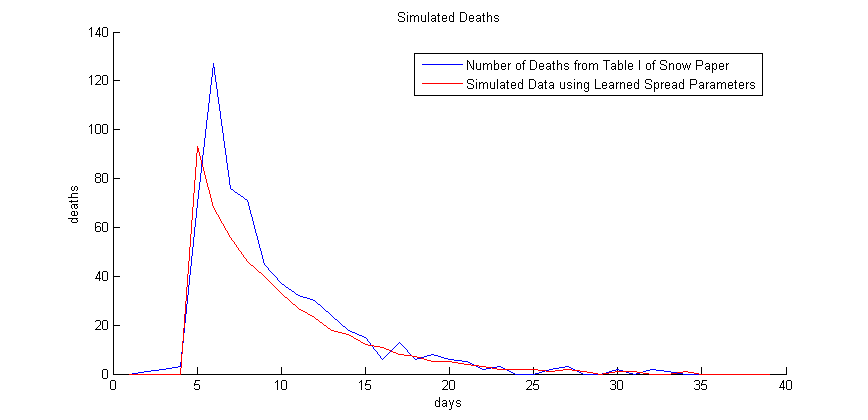}
    \caption{Comparison of Figures \ref{fig:deaths} and \ref{fig:deaths_sim}: Note that there is a difference in the magnitude, but the general shapes are very similar. This discrepancy is due to the fact that we used the spatial dataset in Figures \ref{fig:snow}-\ref{fig:snowdata}, which had only 489 documented deaths, while the cumulative data from Table I in \cite{snow1855mode}, shown in Figure \ref{fig:deaths} and the blue in this plot, has a total of 616 deaths. The difference of 127 has caused the discrepancy.}
    \label{fig:comp}
\end{figure}

For the third attempt, we changed to heterogeneous household sizes, using Figure 1 in \cite{shiode2015mortality} to approximate these values. We also removed all edges except the self loops and the binary directed edges from the pump to every household with at least one death. The connection from the pump to the workhouse was set to $\frac{1}{10}$ (corresponding to the 208th index) because they had their own well and only a small fraction of the 403 residents drank from the Broad Street pump \cite{snow1855mode}. 
Therefore 
\begin{equation}
A_3 = \begin{bmatrix}
1 & 0 &\hdots & 0 & 1 \\
0 & 1 &\hdots & 0 & \vdots \\
0 & 0 &\ddots & 0 & \frac{1}{10} \\\vspace{-.2ex}
0 & 0 &\hdots & 1 & \vdots \\
0 & 0 &\hdots & 0 & 1
\end{bmatrix}.\label{eq:A3}
\end{equation}
After deriving the $\delta_i$ values using Corollary \ref{cor:xs}, we were able to simulate the system using \eqref{eq:dis} with $h=\frac{1}{30}$. The distribution of the deaths is shown in Figure \ref{fig:deaths_sim_best}. As a result of the larger $h$ value, no aggregation of the data was required; the plot shows the complete dataset (with several padded zeros at the beginning and one at the end). 
For completeness, we include 
a link to a video of this simulation in the caption of Figure~\ref{fig:deaths_sim_best}. 

\begin{figure}
    \includegraphics[width=\columnwidth]{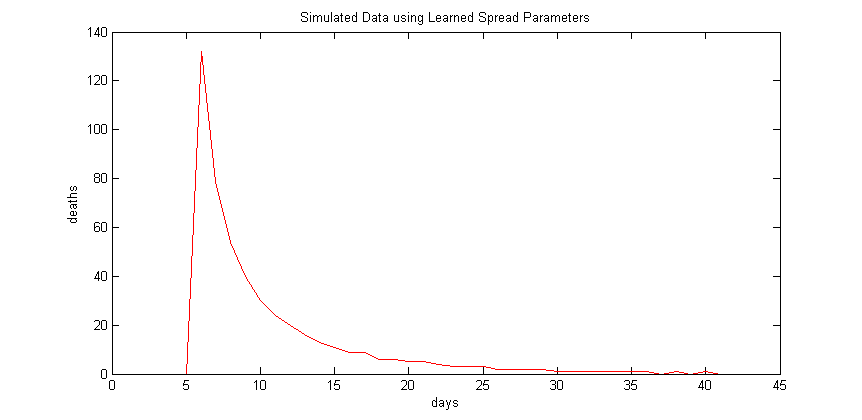}
    \caption{Simulated data using the learned parameters from the data in Figure \ref{fig:snowdata}, employing Corollary \ref{cor:xs} and $A^3$ from~\eqref{eq:A3}. A video of the spread of the simulation can be found at \href{https://youtu.be/oAljxVzyE5U}{youtu.be/oAljxVzyE5U}}.
    \label{fig:deaths_sim_best}
\end{figure}

Plotting the distributions from Figures \ref{fig:deaths} and \ref{fig:deaths_sim_best} on the same plot for comparison in Figure \ref{fig:comp_best} shows that we capture the behavior of the outbreak quite well. 
The lack of the address information for the additional 127 deaths results in the plots not being identical. However, the discrepancy is distributed fairly evenly across all the whole sample time. Consequently, we have shown that the model in \eqref{eq:dis} captures the behavior of the cholera epidemic from John Snow's 1854 dataset very well. Additionally, the fact that $A^3$ from \eqref{eq:A3} performs the best confirms Snow's hypotheses that the Broad Street pump was the source of the cholera outbreak, and that cholera does not spread easily between people or the air, which is known to be true today. 

\begin{figure}
    \includegraphics[width=\columnwidth]{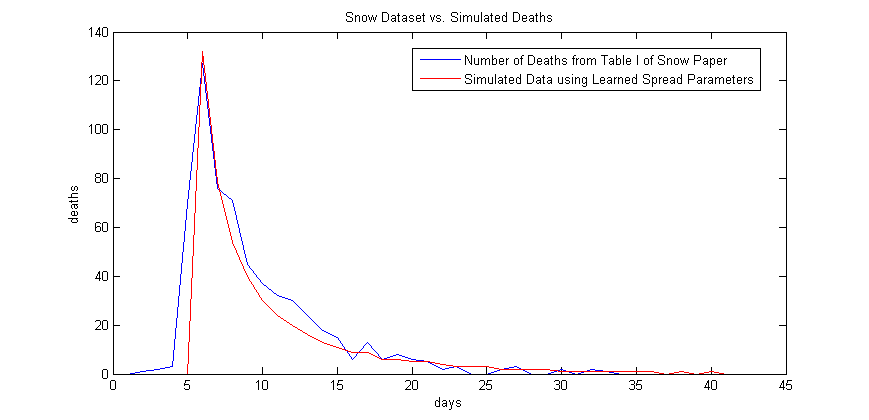}
    \caption{Comparison of Figures \ref{fig:deaths} and \ref{fig:deaths_sim}: Note that there is a difference in the magnitude, but the general shapes are very similar. This discrepancy is due to the fact that we used the spatial dataset in Figures \ref{fig:snow}-\ref{fig:snowdata}, which had only 489 documented deaths, while the cumulative data from Table I in \cite{snow1855mode}, shown in Figure \ref{fig:deaths} and the blue in this plot, has a total of 616 deaths. The difference of 127 has caused the discrepancy.}
    \label{fig:comp_best}
\end{figure} 

\section{Validation: USDA Dataset}\label{sec:usda}
The goal of this section is to study whether variation in the spatial pattern of farmers' enrollment in ACRE during 2009-2012 follows the spreading processes presented in Section~\ref{sec:mod}. 

\subsection{USDA Dataset}\label{sec:data}

The characteristics of the ACRE program make it a good candidate to empirically test the model of spreading. Farmers rely on the experience of neighbors in the adoption of new or complex technologies \cite{foster1995learning,Conley2010a,Duflo2011}. As we elaborate below, ACRE is a complex program. Social and professional networks will likely facilitate the spread of information about the ACRE program from the experiences of early adopters.

The ACRE program was introduced by the Food, Conservation and Energy Act of 2008 (2008 Farm Bill). Initial enrollment was unexpectedly low, in part because of the program's complexity \cite{Rejesus2013}. The ACRE payment $a_{ij}^k$ for year $k$ is calculated by the following formula: 
\begin{align}\nonumber
  a_{ij}^k &= \phi \frac{\hat{y}^k_{ij}}{\hat{y}^k_{\sigma j}}\min\{ (g^k_{\sigma j}-r^k_{\sigma j}), \frac{g^k_{\sigma j}}{4}\}  \min\{\rho_{ij}^k , b_{ij}^k\} 
  \mathds{1}(r^k_{ij} < g^k_{ij})\mathds{1}(r^k_{\sigma j} < g^k_{\sigma j}),\label{eq:acre}
\end{align} 
where $i$ is the farm index; $j$ is the crop or commodity that subsidy corresponds to; $\sigma$ indicates the state (e.g., Idaho); the benchmark yield (a.k.a. the Olympic yield) is
\begin{equation*}
  \hat{y}^k_{\iota j} = \frac{1}{3}\left[ \sum_{l=1}^5 y_{\iota j}^{k-l} - \max\{\Upsilon_{\iota j}\} - \min\{\Upsilon_{\iota j}\}\right],
\end{equation*}
where the set $\Upsilon_{\iota j} = \{y_{\iota j}^{k-1},\dots, y_{\iota j}^{k-5}\}$, for $\iota \in \{ i, \sigma \}$; the farm and state guaranteed revenues per acre are $g^k_{ij}=\hat{y}^k_{ij} {\overline{\overline{p}}}_j^k$ and $g^k_{\sigma j}=.9\hat{y}^k_{\sigma j} {\overline{\overline{p}}}_j^k$, respectively, with ${\overline{\overline{p}}}_j^k = \frac{1}{2}\displaystyle\sum_{l=1}^2 \bar{p}_j^{k-l}$, where $\bar{p}_j^{k}$ is the National Average Market Price of crop $j$
; actual revenue per acre is $r_{\iota j}^k = {y}^k_{\iota j} q_j^k$, with $q_j^k = \max\{0.7l_j^k,\bar{p}_j^{k}\}$, where $l_j^k$ is the National Loan Rate, which Congress sets in the farm bill; $\rho_{ij}^k$ is the number of acres \textit{planted} with crop $j$ on farm $i$; $b_{ij}^k$ is the number of acres of crop $j$ on farm $i$ qualifying for the Direct and Counter-cyclical Payment (DCP) subsidy, which are known as base acres; and $\mathds{1}(\cdot )$ is the indicator function.

The ACRE program benefits farmers by paying out when the farmers' \textit{actual revenue} is low. In contrast, the Counter-cyclical Program (CCP), which ACRE replaces, takes into account current prices but the payout is determined by the subsidized land's productivity in the early 1980s.

The cost to participate in ACRE is not trivial. By choosing ACRE, farmers must forgo 20\% of their annual unconditional subsidy, i.e., Direct Payment, and 30\% of the production subsidy they would receive in the event of low crop prices. Another important consideration is the decision to participate in ACRE is irreversible. Although farmers must   re-enroll in ACRE every year, they cannot switch back to the CCP. Failure to enroll disqualifies farmers from the benefits of ACRE but not the costs. Since switching from ACRE back to CCP is not allowed, we should expect the healing rate $\delta$ to be small compared to the infection rate $\beta$, when we learn the model parameters from the data.

The dataset includes the total annual payments received by each farm in the U.S. for each USDA-sponsored program from the year 2008 to 2012. Each datapoint has a program, payment amount, payment date, contract number, commodity (usually the crop), the farm number, and the customer's (farmer's) identification number and address. 
The dataset allows for the possibility to investigate the spread of the ACRE program through several different networks. 
Farmer-to-farmer networks could be created from the data by connecting farmer-nodes who receive payments on the same field or live nearby. 
Alternatively, farms can be aggregated to the county level. This approach allows us to convert the binary decision to enroll in ACRE into a continuous measure of the proportion of eligible farms that enroll in ACRE in each county. The proportion of farms enrolled in ACRE corresponds exactly to the density of infection, facilitating our investigation of the spread of ACRE.  For counties where no farms are enrolled in either, the infection state is set to zero. Alaska and Hawaii are omitted. The data for the four years considered can be found in Figures \ref{fig:data09}-\ref{fig:data12}.

\begin{figure}
  \begin{multicols}{2}
  \centering
  \begin{subfigure}{\linewidth}
    \includegraphics[width=\linewidth]{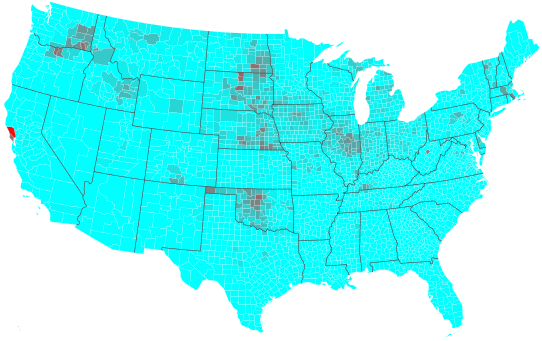}
    \caption{2009 Data}
    \label{fig:data09}
  \end{subfigure}
  \par
  \begin{subfigure}[b]{\columnwidth}
    \includegraphics[width=\columnwidth]{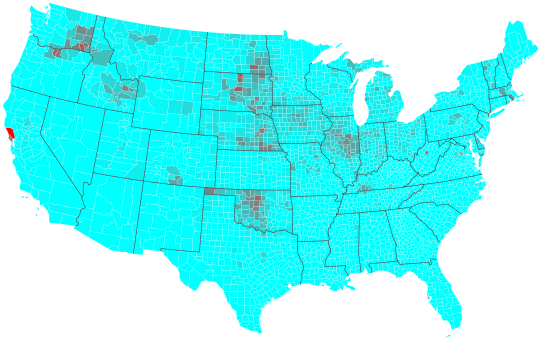}
    \caption{2010 Data}
    \label{fig:data10}
  \end{subfigure}
  \par
  \begin{subfigure}[b]{\columnwidth}
    \includegraphics[width=\columnwidth]{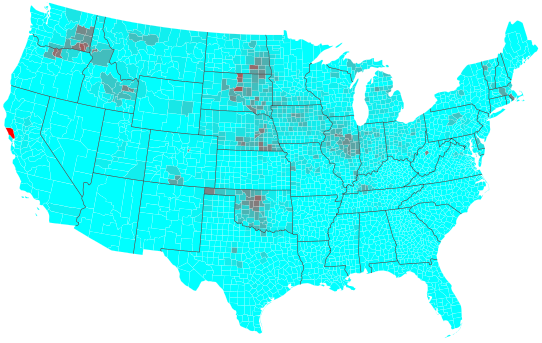}
    \caption{2011 Data}
    \label{fig:data11}
  \end{subfigure}
  \par
  \begin{subfigure}[b]{\columnwidth}
    \includegraphics[width=\columnwidth]{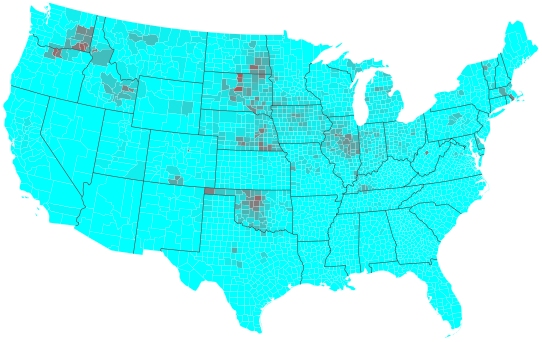}
    \caption{2012 Data}
    \label{fig:data12}
  \end{subfigure}
  \par
    \begin{subfigure}{\columnwidth}
      \includegraphics[width=\columnwidth]{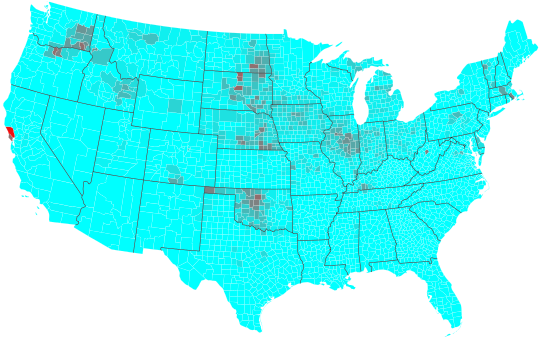}
      \caption{2009 Simulated Data}
      \label{fig:result09}
    \end{subfigure}
    \par
    \begin{subfigure}[b]{\columnwidth}
      \includegraphics[width=\columnwidth]{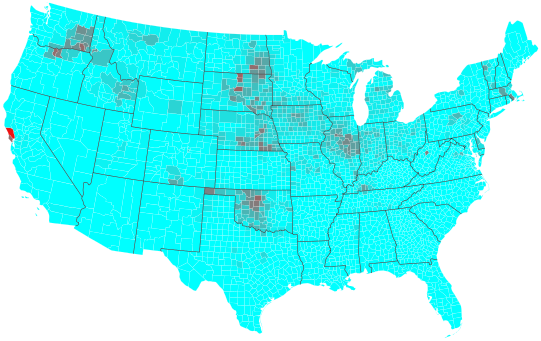}
      \caption{2010 Simulated Data}
      \label{fig:result10}
    \end{subfigure}
    \par
    \begin{subfigure}[b]{\columnwidth}
      \includegraphics[width=\columnwidth]{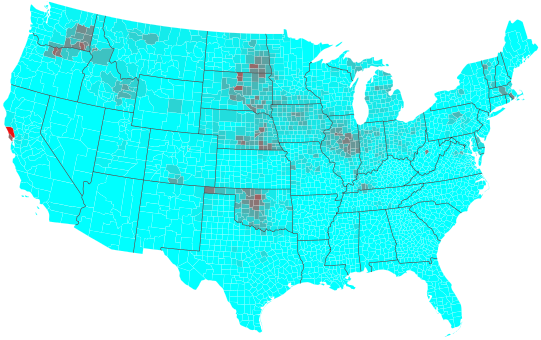}
      \caption{2011 Simulated Data}
      \label{fig:result11}
    \end{subfigure}
    \par
    \begin{subfigure}[b]{\columnwidth}
      \includegraphics[width=\columnwidth]{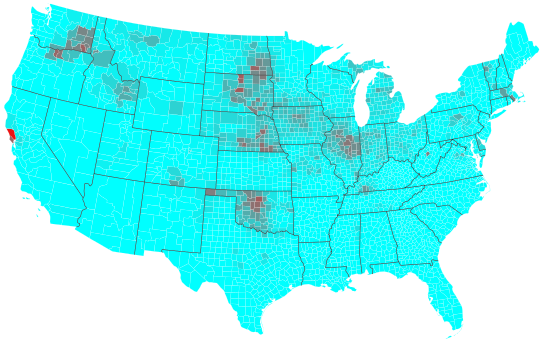}
      \caption{2012 Simulated Data}
      \label{fig:result12}
    \end{subfigure}
    \end{multicols}
  \caption{(Left) The percentage of farms enrolled in the ACRE Program that are enrolled in either ACRE or CCP calculated from the USDA dataset. (Right) Simulated data using Figure \ref{fig:data09} as the initial condition on the model in \eqref{eq:dis} with parameters calculated using the data from Idaho, given in \eqref{eq:idaho}.}
  \label{fig:data}
\end{figure}

\subsection{USDA Farm Subsidies as a Spread Process} \label{subsec:usda}

Now we use the learning techniques presented in Section~\ref{sec:id} and tested in Section~\ref{sec:sim} for the model in \eqref{eq:dis} on the data presented in Section~\ref{sec:usda}. 
We learn the homogeneous model parameters using a subset of the dataset, the USDA data from Idaho, and then simulate the spread of ACRE over the whole contiguous United States using the learned parameters. The adjacency matrices are calculated using  the adjacency of counties, that is, 
\begin{equation}
    a_{ij} = 
    \begin{cases}
    1, & \text{ if county }i \text{ and county } j \text{ share a border,}\\
    1, &\text{ if } i=j\text{,} \\
    0, & \text{ otherwise.}
    \end{cases}\label{eq:astate}
\end{equation} 
For calculating the adjacency matrix for Idaho, adjacent counties from bordering states were ignored. 
Applying \eqref{eq:id1} on the Idaho dataset, with $h=1$, gives the following spread parameters:
\begin{equation}
     \begin{bmatrix}  \hat{\delta} \\ \hat{\beta} \end{bmatrix} = \begin{bmatrix}   0.00909176 \\ 0.02237450 \end{bmatrix}.
    \label{eq:idaho}
\end{equation}
As expected, switching $h$ to the value $.1$ moves the decimal point one place to the right. 

To validate the model, we simulate the spread over the contiguous United States using the model in \eqref{eq:dis} with parameters calculated using the data from Idaho, given in \eqref{eq:idaho}, with the data from Figure \ref{fig:data09} being used as the initial condition. The simulation results are given in Figures \ref{fig:result09}-\ref{fig:result12}. The scaled error between the dataset, $\mathbb{F}$, and the simulated data, $\hat{\mathbb{F}}$, using the Frobenius norm is 
\begin{equation*}
    \frac{\left\| \mathbb{F} - \hat{\mathbb{F}} \right\|_{F}}{\left\| \mathbb{F} \right\|_{F}} = \frac{2.5331}{10.7872} = 0.2348.
\end{equation*}
While the model does not perfectly fit the data, it does seem to give some insight into the behavior of the system.
Therefore, if the USDA wanted to test a pilot program in a certain region of the country, for example Idaho, the resulting behavior could give some insight into how the whole country would react. The four time steps (years) does not allow the system to reach the equilibrium state, so the behavior depends significantly on the initial condition. Therefore, given the model learned from a pilot program, the USDA could determine the best counties to target for advertising of the new subsidy programs, assuming they wanted to maximize adoption of the new program. 




\section{Conclusion}\label{sec:con}


We have investigated the relationship between several different spread models. 
We have provided necessary and sufficient conditions for uniqueness of the healthy equilibrium, and proved the existence of an endemic state under certain conditions. We have also provided a necessary condition for asymptotic stability of the healthy state.  
We have presented necessary and sufficient conditions for learning discrete-time spread models from data. 
We have validated a discrete-time SIS virus spread model using John Snow's seminal cholera dataset with very good results. We have also used a USDA dataset to validate the same model by modeling the spread of farming subsidies among farms/farmers aggregated by county. 

In future work, we would like to provide further analysis on the endemic state of the system. 
We would like to further study identification of the spread model allowing noise in the data. 
We would also like to find other datasets to help further validate the SIS spread models. 
Finally, we would like to employ the results herein to develop effective control techniques to mitigate the spread of disease in real systems. 


\section{Acknowledgement}


The authors wish to thank Aditya Shivashankar (UIUC) for helping with the data processing of the Snow dataset.


\bibliographystyle{IEEEtran}
\bibliography{IEEEabrv,bib}

\end{document}